\documentclass[12pt, reqno]{amsart}
\usepackage{amsmath, amsthm, amscd, amsfonts, amssymb, graphicx, color}
\usepackage[bookmarksnumbered, colorlinks, plainpages]{hyperref}
\hypersetup{colorlinks=true,linkcolor=red, anchorcolor=green, citecolor=cyan, urlcolor=red, filecolor=magenta, pdftoolbar=true}

\newtheorem{theorem}{Theorem}[section]
\newtheorem{lemma}[theorem]{Lemma}

\newtheorem{cor}[theorem]{Corollary}
\theoremstyle{definition}

\theoremstyle{remark}
\newtheorem{remark}[theorem]{\bf{Remark}}
\numberwithin{equation}{section}
\begin{document}

\title [Refinements of norm and numerical radius inequalities ] {  Refinements of norm and numerical radius inequalities  }

\author[P. Bhunia and K. Paul]{Pintu Bhunia and Kallol Paul}

\address{(Bhunia) Department of Mathematics, Jadavpur University, Kolkata 700032, West Bengal, India}
\email{pintubhunia5206@gmail.com}

\address{(Paul) Department of Mathematics, Jadavpur University, Kolkata 700032, West Bengal, India}
\email{kalloldada@gmail.com;kallol.paul@jadavpuruniversity.in}

\thanks{First author  would like to thank UGC, Govt. of India for the financial support in the form of SRF. Prof. Kallol Paul would like to thank RUSA 2.0, Jadavpur University for the partial support.}
\thanks{}
\thanks{}



\maketitle

\begin{abstract}
Several refinements of norm and numerical radius inequalities of bounded linear operators on a complex Hilbert space are given. In particular, we show that if $A$ is a bounded linear operator on a complex Hilbert space, then
$$	\frac{1}{4}\|A^*A+AA^*\| 
\leq  \frac{1}{8}\bigg( \|A+A^*\|^2+\|A-A^*\|^2 +c^2(A+A^*)+c^2(A-A^*)\bigg) \leq w^2(A)$$
and 
\begin{eqnarray*}
	\frac{1}{2}\|A^*A+AA^*\| -  \frac{1}{4}\bigg\|(A+A^*)^2 (A-A^*)^2 \bigg\|^{1/2} \leq w^2(A)  \leq 	\frac{1}{2}\|A^*A+AA^*\|,
\end{eqnarray*}
where $\|.\|$, $w(.)$ and $c(.)$ are the operator norm, the numerical radius and the Crawford number, respectively. 
Further, we prove that  if $A,D$ are bounded linear operators on a complex Hilbert space, then
\begin{eqnarray*}
\|AD^*\| \leq  \left\| \int_0^1  \left( (1-t) \left(\frac{ |A|^2+|D|^2}{2}\right) +t\|AD^*\|I \right)^2dt  \right\|^{1/2} \leq  \frac{1}{2}\left\| |A|^2+|D|^2 \right\|,
\end{eqnarray*}
where $|A|=(A^*A)^{1/2}$ and $|D|=(D^*D)^{1/2}$.
This is a refinement of well known inequality obtained by Bhatia and Kittaneh.



\end{abstract}

\section{\textbf{Introduction and preliminaries}}

\noindent Let  $ \mathbb{H}$ denote a complex Hilbert space with inner product  $\langle.,.\rangle$ and $ \|.\|$ denotes the norm induced by the inner product $\langle.,.\rangle$.  Let $ \mathcal{B}(\mathbb{H}) $ denote the collection of all bounded linear operators on $ \mathbb{H}.$ 
For  $A \in \mathcal{B}(\mathcal{H})$, $A^*$ denotes the adjoint of $A$ and $|A|=\sqrt{A^*A}$. For  $A \in \mathcal{B}(\mathcal{H})$, let $\|A\|$ be the operator norm of $A$. Recall that 
\[\|A\|:=\sup_{\|x\|=1}\|Ax\|=\sup_{\|x\|=\|y\|=1}|\langle Ax,y\rangle|.\] The numerical range of $A$, denoted by $W(A),$ is defined as  \[W(A):=\big\{ \langle Ax,x\rangle~~:~~x\in \mathcal{H}, ~~\|x\|=1 \big\}.\]
The  two numerical constants, numerical radius $w(A)$  and  Crawford number $c(A)$,  associated with the numerical range  of $A$,  are defined respectively as 
\begin{eqnarray*}
w(A):= \sup_{\lambda \in W(A)} |\lambda| 
~~\,\,\,\,\mbox{and}~~\,\,\,\,
c(A):=\inf_{\lambda \in W(A)} |\lambda|.
\end{eqnarray*}
The numerical radius is a norm on $\mathcal{B}(\mathbb{H}) $ satisfying the following inequality 
\begin{eqnarray}\label{eqv}
\frac{\|A\|}{2}\leq w(A)\leq \|A\|.
\end{eqnarray}
Clearly, (\ref{eqv}) implies that the numerical radius norm is equivalent to the operator norm. The inequality  (\ref{eqv}) is sharp,  $w(A) = \|A\|$ if $AA^*=A^*A$ and $ w(A)=\frac{\|A\|}{2} $ if $A^2=0.$ 

\noindent Various refinemets of (\ref{eqv}) have been obtained, we refer to \cite{bag, BBP3, BPN,OM} and references therein.
Kittaneh \cite[Th. 1]{kittaneh1} improved on the inequality (\ref{eqv}), obtained that
\begin{eqnarray}\label{imp1}
\frac{1}{4}\|A^*A+AA^*\| &\leq&  w^2(A) \leq \frac{1}{2}\|A^*A+AA^*\|.
\end{eqnarray}
In \cite[Th. 1]{K03}, Kittaneh obtained another refinement of the second inequality in (\ref{eqv}), he proved that
\begin{eqnarray}\label{imp2}
 w(A) &\leq&  \frac{1}{2}\||A|+|A^*|\| \leq \frac{1}{2}\|A\|+\frac{1}{2}\|A^2\|^{1/2}.
\end{eqnarray}
Further refinements of (\ref{imp1}) and (\ref{imp2}) have obtained in \cite{BP}, namely,
\begin{eqnarray}\label{imp3}
w^2(A)& \leq &  \min_{0\leq \alpha \leq 1} \left \| \alpha A^*A +(1-\alpha)AA^* \right \|
\end{eqnarray} 
and
\begin{eqnarray}\label{imp4}
w^{2}(A)&\leq &\min \{ \gamma_1, \gamma_2\},
\end{eqnarray}
 \mbox{where }
$$ \gamma_1 = \min_{0\leq \alpha \leq 1} \left\| \alpha   \left(\frac{|A|+|A^*|}{2}\right)^{2}+\left(1-\alpha\right)  |A^*|^{2}  \right\|, $$ 
$$ \gamma_2 = \min_{0\leq \alpha \leq 1} \left\| \alpha   \left(\frac{|A|+|A^*|}{2}\right)^{2}+\left(1-\alpha\right)  |A|^{2}  \right\|.$$\\
In \cite[Th. 2.1]{OM}, Omidvar and Moradi have obtained the following inequality
\begin{eqnarray}\label{moradi}
\frac{1}{4} \|A+A^*\|\|A-A^*\|&\leq &  w^2(A).
\end{eqnarray}

\noindent In this paper, we obtain various norm inequalities for the sum of two bounded linear operators which improve on the triangle inequality. As an application of the norm inequalities we obtain several numerical radius inequalities which improve on the first inequality in (\ref{imp1}). For $A\in \mathcal{B}(\mathcal{H}),$ we prove that
\begin{eqnarray*}
 \frac{1}{4}\|A^*A+AA^*\| &\leq&   \frac{1}{8}\big( \|A+A^*\|^2+\|A-A^*\|^2\big) \\
&  \leq & \frac{1}{8}\big( \|A+A^*\|^2+\|A-A^*\|^2\big) +\frac{1}{8}c^2\big(A+A^*\big)+\frac{1}{8}c^2\big(A-A^*\big)\\
&\leq&  w^2(A).
\end{eqnarray*}
The above inequality is stronger than the inequality (\ref{moradi}). Another refinements of (\ref{moradi}) also given.


\noindent The notion of operator convex function plays an important role in the development of operator and norm inequalties. A real-valued continuous function $f$ on an interval $J$ is said to be operator convex if for all selfadjoint operators $A, D \in \mathcal{B}(\mathcal{H})$ whose spectra are contained in $J$ satisfy $$f \big((1 - t)A + tD\big) \leq (1 - t) f(A) + tf(D)$$ for all $t \in [0, 1]$. The function $f(t) = t^r$ is operator convex on $[0,\infty)$ if either $1\leq r\leq 2$ or $-1\leq r\leq 0$. 
In \cite{BK}, Bhatia and Kittaneh have obtained a norm inequality,  for $A,D\in \mathcal{B}(\mathcal{H})$, 
\begin{eqnarray}\label{300c}
	\|AD^*\| &\leq&   \frac{1}{2}\left\| A^*A+D^*D \right\|.
\end{eqnarray}
We obtain a refinement of the above inequality to show that 
\begin{eqnarray*}
f(\|AD^*\|) &\leq&  \left\| \int_0^1 f  \left( (1-t)\left(\frac{|A|^2+|D|^2}{2}\right) +t\|AD^*\|I \right)dt  \right\| \\
&\leq&  \left\| f\left(\frac{A^*A+D^*D}{2}\right) \right\|.
\end{eqnarray*}
In particular, considering $ f(t) = t^2,$ we get
\begin{eqnarray*}
\|AD^*\| &\leq & \frac{1}{\sqrt{3}}\left\| \left(\frac{|A|^2+|D|^2}{2}\right)^2 +\|AD^*\|^2I + \|AD^*\|\left(\frac{|A|^2+|D|^2}{2}\right) \right\|^{1/2} \\
&\leq & \frac{1}{2}\left\| A^*A+D^*D \right\|.
\end{eqnarray*}
This is a non-trivial improvement of Bhatia and Kittaneh's inequality (\ref{300c}).
In \cite{BK2}, Bhatia and Kittaneh have obtained that if $A,D\in \mathcal{B}(\mathcal{H})$ be positive then
\begin{eqnarray}\label{bhatia}
\|AD\| &\leq & \frac{1}{4} \|A+D\|^2.
\end{eqnarray}
We obtained a refinement of the above inequality (\ref{bhatia}).
The refinements of (\ref{imp3}) and (\ref{imp4}) are also obtained.

 \section{\textbf{Norm inequalities in estimating  lower bound of numerical radius} }

\noindent We begin this section  with the introduction of  two notations. Let $A=B+{\rm i} C$ be the Cartesian decomposition of $A$ where $B=\textit{Re}(A):=\frac{A+A^*}{2}$ and $C=\textit{Im}(A):=\frac{A-A^*}{2{\rm i}}.$ We observe that 
\begin{eqnarray}\label{eql1}
\frac{1}{4}\|A^*A+AA^*\|= \frac{1}{2}\|B^2+C^2\|.
\end{eqnarray}
By using the identity (\ref{eql1}), we obtain our first refinement.

\begin{theorem}\label{th1}
Let $A\in \mathcal{B}(\mathcal{H}).$ Then
\begin{eqnarray*}
\frac{1}{4}\|A^*A+AA^*\| &\leq & \frac{1}{8}\big( \|A+A^*\|^2+\|A-A^*\|^2\big) \\
&\leq & \frac{1}{8}\big( \|A+A^*\|^2+\|A-A^*\|^2\big) +\frac{1}{8}c^2\big(A+A^*\big)+\frac{1}{8}c^2\big(A-A^*\big)\\
&\leq& w^2(A).
\end{eqnarray*}

\end{theorem}

\begin{proof}
From the identity (\ref{eql1}) we get,
\begin{eqnarray*}
\frac{1}{4}\|A^*A+AA^*\| = \frac{1}{2}\|B^2+C^2\| \leq  \frac{1}{2}(\|B\|^2+\|C\|^2)=\frac{1}{8}( \|A+A^*\|^2+\|A-A^*\|^2).
\end{eqnarray*}
This is the first inequality of the theorem. The second inequality follows trivially. Now we prove the third inequality. Let $x\in \mathcal{H}$ with $\|x\|=1$. Then from the Cartesian decomposition of $A$ we get,
\begin{eqnarray}\label{eqn1}
|\langle Bx,x\rangle|^2+|\langle Cx,x\rangle|^2=|\langle Ax,x\rangle|^2.
\end{eqnarray}
From (\ref{eqn1}), we get the following two inequalities
\begin{eqnarray}\label{eqn2}
c^2(B)+\|C\|^2\leq w^2(A)
\end{eqnarray}
and 
\begin{eqnarray}\label{eqn3}
c^2(C)+\|B\|^2\leq w^2(A).
\end{eqnarray}
It follows from the inequalities (\ref{eqn2}) and (\ref{eqn3}) that 
\begin{eqnarray*}
c^2(B)+c^2(C)+\|B\|^2+\|C\|^2 \leq 2 w^2(A).
\end{eqnarray*}
This implies that 
\begin{eqnarray*}
\frac{1}{8}\bigg( \|A+A^*\|^2+\|A-A^*\|^2\bigg) +\frac{1}{8}c^2\big(A+A^*\big)+\frac{1}{8}c^2\big(A-A^*\big)\leq w^2(A).
\end{eqnarray*}
This completes the proof.
\end{proof}




\begin{remark} 
We note that the  inequalities obtained  in Theorem \ref{th1} refine  the inequality (\ref{moradi}) obtained by Omidvar and Moradi \cite[Th. 2.1]{OM} and  the  inequality (\ref{imp1}) obtained by Kittaneh \cite[Th. 1]{kittaneh1}. 
 
\end{remark}

The following corollary follows from Theorem \ref{th1}. 

\begin{cor}\label{cor50}
Let $A\in \mathcal{B}(\mathcal{H})$. Then
\begin{eqnarray}\label{weak}
\frac{1}{4}\|A^*A+AA^*\|+\frac{1}{8}c^2\big(A+A^*\big)+\frac{1}{8}c^2\big(A-A^*\big)\leq w^2(A).
\end{eqnarray}
\end{cor}

It should be mentioned here that the inequality (\ref{weak}) is weaker than the third inequality in Theorem \ref{th1}.\\

In the next theorem we obtain a norm inequality which refines the triangle inequality.

\begin{theorem}\label{th7}
Let $A, D\in \mathcal{B}(\mathcal{H}).$ Then
\begin{eqnarray*}
\|A+D\|^2 &\leq & \|A\|^2+\|D\|^2+ \|A^*D+D^*A\|\leq (\|A\|+\|D\|)^2.
\end{eqnarray*}

\end{theorem}

\begin{proof}
Let $x\in \mathcal{H}$ with $\|x\|=1$. Then we have,
\begin{eqnarray*}
\|(A+D)x\|^2 &=& \langle (A+D)x,(A+D)x\rangle\\
&=& \|Ax\|^2+\|Dx\|^2+\langle (A^*D+D^*A)x,x\rangle\\
&\leq& \|A\|^2+\|D\|^2+\|A^*D+D^*A\|.
\end{eqnarray*}
Taking supremum over $\|x\|=1$ we get the first inequality of the theorem. The second inequality follows from the inequality $\|A^*D+D^*A\|\leq 2\|A\|\|D\|.$
\end{proof}

\begin{remark}
We would like to note that if $\|A+D\|=\|A\|+\|D\|$ then it follows from Theorem \ref{th7} that $\|A^*D+D^*A\|=2\|A\|\|D\|.$ The converse is not true, in general.  Consider $A=I$ and $D=-I$,  then $\|A^*D+D^*A\|=2\|A\|\|D\|$, but $\|A+D\|\neq \|A\|+\|D\|.$

\end{remark}

Next we need the following inequality, known as Buzano's inequality.

\begin{lemma} $($\cite{buzano}$)$\label{lemma1}
Let $x,e,y\in \mathcal{H}$ with $\|e\|=1.$ Then 
\[|\langle x,e\rangle~~\langle e,y\rangle|\leq \frac{1}{2}\left(\|x\|~~\|y\|+|\langle x,y\rangle|\right).\]
\end{lemma}

We now obtain another refinement of the triangle inequality.

\begin{theorem} \label{th2}
Let $A, D\in \mathcal{B}(\mathcal{H}).$ Then
\begin{eqnarray*}
\|A+D\|^2 \leq \|A\|^2+\|D\|^2+\|A\| \|D\|+\min \big \{ w(A^*D), w(AD^*) \big \} \leq (\|A\|+\|D\|)^2.
\end{eqnarray*}
\end{theorem}

\begin{proof}
Let $x,y\in \mathcal{H}$ be two unit vectors. Then we get,
\begin{eqnarray*}
|\langle (A+D)x,y \rangle |^2 &\leq & (|\langle Ax,y \rangle |+|\langle Dx,y \rangle |)^2\\
&=& |\langle Ax,y \rangle |^2+|\langle Dx,y \rangle |^2+2|\langle Ax,y \rangle \langle Dx,y \rangle |\\
&=& |\langle Ax,y \rangle |^2+|\langle Dx,y \rangle |^2+2|\langle Ax,y \rangle \langle y,Dx \rangle |\\
&\leq& |\langle Ax,y \rangle |^2+|\langle Dx,y \rangle |^2+\|Ax\| \|Dx\|+ |\langle Ax,Dx \rangle |,\\
&& ~~\,\,\,\,\,\,\,\,\,\,\,\,\,\,\,\,\,\,\,\,\,\,\,\,\,\,\,\,\,\,\,\,\,\,\,\,\,\,\, \,\,\,\,\,\,\,\,\,\,\,\,\,\,\,\,\,\,\,\,\,\,\,\,\,\,\,\,\,\,\,\,\,\,\,\,\,\,\,\,\,\,\,\,\,\,\,\,\,\,\,\,\,\,\,\,\,\,\,\,\,\,\,~~\mbox{By Lemma \ref{lemma1}}\\
&\leq& \|A\|^2+\|D\|^2+\|A\|\|D\|+w(A^*D).
\end{eqnarray*}
Taking supremum over $\|x\|=\|y\|=1$ we get,
\begin{eqnarray}\label{eqn4}
\|A+D\|^2 \leq \|A\|^2+\|D\|^2+\|A\|\|D\|+w(A^*D).
\end{eqnarray}
Replacing $A$ by $A^*$ and $D$ by $D^*$ in (\ref{eqn4}) we get, 
\begin{eqnarray}\label{eqn5}
\|A+D\|^2 \leq \|A\|^2+\|D\|^2+\|A\|\|D\|+w(AD^*).
\end{eqnarray}
Combining (\ref{eqn4}) and (\ref{eqn5}) we have the first inequality of the theorem.  The second inequality follows from the observation that   $\min \{ w(A^*D), w(AD^*)\}\leq \|A\|\|D\|.$ 
\end{proof}

\begin{remark}
It follows from Theorem \ref{th2} that if $\|A+D\|=\|A\|+\|D\|$ then  
$w(A^*D)=\|A\|\|D\|$ and $ w(AD^*)=\|A\|\|D\|.$
The converse is  not true, in general. Consider $A=I$ and $D=-I,$ then $w(A^*D)=\|A\|\|D\|~~\mbox{and}~~w(AD^*)=\|A\|\|D\| $, but $\|A+D\|\neq \|A\|+\|D\|$. 
\end{remark}






Now we need the following norm inequality.

\begin{lemma} $($\cite{DP}$)$\label{lem5}
Let $A, D\in \mathcal{B}(\mathcal{H})$ be positive. Then
\begin{eqnarray*}
\|A+D\| &\leq & \max \{ \|A\|,\|D\|\}+ \|AD\|^\frac{1}{2}.
\end{eqnarray*}
\end{lemma}

Next refinement of Kittaneh's inequality (\ref{imp1}) is as follows.

\begin{theorem}\label{th9}
Let $A\in \mathcal{B}(\mathcal{H}).$ Then
\begin{eqnarray*}
&& \frac{1}{4}\|A^*A+AA^*\| \\
&\leq& \frac{1}{8}\bigg[ \max \big\{ \|A+A^*\|^2,\|A-A^*\|^2 \big\}+\|A+A^*\|\|A-A^*\| \bigg] \leq w^2(A).
\end{eqnarray*}
\end{theorem}

\begin{proof}
From the identity (\ref{eql1}) we get,
\begin{eqnarray*}
\frac{1}{4}\|A^*A+AA^*\| &=& \frac{1}{2}\|B^2+C^2\| \\
&\leq & \frac{1}{2} \bigg[ \max\{ \|B\|^2,\|C\|^2 \} +\| B^2C^2 \|^{\frac{1}{2}} \bigg],~~\mbox{by Lemma \ref{lem5}}\\
&\leq & \frac{1}{2} \bigg[ \max\{ \|B\|^2,\|C\|^2 \} +\| B\| \|C \| \bigg].
\end{eqnarray*}
This implies the first inequality of the theorem. The second inequality follows from the observation that  $\|B\|\leq w(A)$ and $\|C\|\leq w(A)$.
\end{proof}


\begin{remark}
We  note that the second inequality in Theorem \ref{th9} refines the inequality (\ref{moradi}), obtained by Omidvar and Moradi  \cite[Th. 2.1]{OM}.
\end{remark}

Next we need the following lemma.

\begin{lemma} $($\cite[Th. 2.2]{BBP4}$)$ \label{lemma2}
Let $A, D\in \mathcal{B}(\mathcal{H}).$ Then
\begin{eqnarray*}
\|A+D\|^2 &\leq & 2 \max\big\{ \|A^*A+D^*D\|, \|AA^*+DD^*\| \big\}.
\end{eqnarray*}
\end{lemma}

Based on the above lemma we obtain another refinement of Kittaneh's inequality (\ref{imp1}).

\begin{theorem}\label{th6a}
Let $A\in \mathcal{B}(\mathcal{H}).$ Then
\begin{eqnarray*}
\frac{1}{4}\|A^*A+AA^*\| &\leq & \frac{1}{4\sqrt{2}} \bigg[\big\|A+A^*\big\|^4+\big \|A-A^*\big\|^4\bigg]^{\frac{1}{2}} \leq w^2(A).
\end{eqnarray*}

\end{theorem}

\begin{proof}
From the identity (\ref{eql1}) we get,
\begin{eqnarray*}
\frac{1}{4}\|A^*A+AA^*\| &=& \frac{1}{2}\|B^2+C^2\| \\
&\leq & \frac{1}{\sqrt{2}}\|B^4+C^4\|^{\frac{1}{2}},~~\mbox{by Lemma \ref{lemma2}}.\\
&\leq & \frac{1}{\sqrt{2}}\bigg[\|B\|^4+\|C\|^4\bigg]^{\frac{1}{2}}.
\end{eqnarray*}
This implies the first inequality of the theorem. As before, the second inequality follows from the observation that  $\|B\|\leq w(A)$ and $\|C\|\leq w(A)$.
\end{proof}

\begin{remark}
The concavity of the function $f(t)=\sqrt{t}$ ensures that the first inequality in Theorem \ref{th6a} is stronger than the first inequality in Theorem \ref{th1}. We also note that the second inequality in Theorem \ref{th6a} refines the inequality (\ref{moradi}), obtained by Omidvar and Moradi \cite[Th. 2.1]{OM}.
\end{remark}

To obtain the next  refinement of (\ref{imp1}), we need the following lemma.

\begin{lemma} $($\cite[Th. 2.5]{BBP4}$)$ \label{lemma3}
Let $A, D\in \mathcal{B}(\mathcal{H}).$ Then
\begin{eqnarray*}
\|A+D\|^4 \leq  2 \max\big\{ \|A^*A+D^*D\|^2+4w^2(D^*A), \|AA^*+DD^*\|^2+4w^2(AD^*) \big\}.
\end{eqnarray*}
\end{lemma}

\begin{theorem}\label{th6}
Let $A\in \mathcal{B}(\mathcal{H}).$ Then
\begin{eqnarray*}
&& \frac{1}{4}\|A^*A+AA^*\| \\
&\leq&  \frac{1}{8} \bigg[2 \big ( \|A+A^*\|^4+\|A-A^*\|^4 \big)^2 +8 \|A+A^*\|^4 \|A-A^*\|^4  \bigg]^{\frac{1}{4}} \leq w^2(A).
\end{eqnarray*}

\end{theorem}

\begin{proof}
From the identity (\ref{eql1}) we get,
\begin{eqnarray*}
\frac{1}{4}\|A^*A+AA^*\| &=& \frac{1}{2}\|B^2+C^2\| \\
&\leq & \frac{1}{2} \bigg[ 2\|B^4+C^4\|^2 +8\max \{ w^2(B^2C^2), w^2(C^2B^2) \}\bigg]^{\frac{1}{4}},\\
&&\,\,\,\,\,\,\,\,\,\,\,\,\,\,\,\,\,\,\,\,\,\,\,\,\,\,\,\,\,\,\,\,\,\,\,\,\,\,\,\,\,\,\,\,\,\,\,\,\,\,\,\,\,\,\,\,\,\,\,\,\,\,\,\,\,\,\,\,\,\,\,\,\,\,\,\,\,\,\,\,\,\,\,\,\,\,\,\,\,\,\,\,\,\,\,\,~~\mbox{by Lemma \ref{lemma3}}\\
&\leq & \frac{1}{2} \bigg[ 2(\|B\|^4+\|C\|^4)^2 +8 \|B\|^4\|C\|^4\bigg]^{\frac{1}{4}}.
\end{eqnarray*}
This implies the first inequality of the theorem. As before, the second inequality follows from the observation that  $\|B\|\leq w(A)$ and $\|C\|\leq w(A)$.
\end{proof}

We now prove the following norm inequalities.

\begin{theorem} \label{th4}
Let $A, D\in \mathcal{B}(\mathcal{H}).$ Then the following inequalities hold:
\begin{eqnarray}\label{eqn6}
\|A+D\|^2 \leq \|A\|^2+\|D\|^2+\frac{1}{2}\|A^*A+D^*D\|+w(A^*D)
\end{eqnarray}
and
\begin{eqnarray}\label{eqn7}
\|A+D\|^2 \leq \|A\|^2+\|D\|^2+\frac{1}{2}\|AA^*+DD^*\|+w(AD^*).
\end{eqnarray}

\end{theorem}

\begin{proof}
Let $x,y\in \mathcal{H}$ be two unit vectors. Then we get,
\begin{eqnarray*}
|\langle (A+D)x,y \rangle |^2 &\leq & (|\langle Ax,y \rangle |+|\langle Dx,y \rangle |)^2\\
&=& |\langle Ax,y \rangle |^2+|\langle Dx,y \rangle |^2+2|\langle Ax,y \rangle \langle Dx,y \rangle |\\
&=& |\langle Ax,y \rangle |^2+|\langle Dx,y \rangle |^2+2|\langle Ax,y \rangle \langle y,Dx \rangle |\\
&\leq& |\langle Ax,y \rangle |^2+|\langle Dx,y \rangle |^2+\|Ax\| \|Dx\|+ |\langle Ax,Dx \rangle |,\\
&& ~~\,\,\,\,\,\,\,\,\,\,\,\,\,\,\,\,\,\,\,\,\,\,\,\,\,\,\,\,\,\,\,\,\,\,\,\,\,\,\, \,\,\,\,\,\,\,\,\,\,\,\,\,\,\,\,\,\,\,\,\,\,\,\,\,\,\,\,\,\,\,\,\,\,\,\,\,\,\,\,\,\,\,\,\,\,\,\,\,\,\,\,\,\,\,\,\,\,\,\,\,\,\,~~\mbox{By Lemma \ref{lemma1}}\\
&\leq& |\langle Ax,y \rangle |^2+|\langle Dx,y \rangle |^2+\frac{1}{2}(\|Ax\|^2+ \|Dx\|^2)+ |\langle Ax,Dx \rangle |\\
&\leq& |\langle Ax,y \rangle |^2+|\langle Dx,y \rangle |^2\\
&& +\frac{1}{2}\langle (A^*A+D^*D)x,x \rangle + |\langle A^*Dx,x \rangle |\\
&\leq& \|A\|^2+\|D\|^2+\frac{1}{2}\|A^*A+D^*D\|+w(A^*D).
\end{eqnarray*}
Taking supremum over $\|x\|=\|y\|=1$ we get,
\begin{eqnarray*}
\|A+D\|^2 \leq \|A\|^2+\|D\|^2+\frac{1}{2}\|A^*A+D^*D\|+w(A^*D).
\end{eqnarray*}
Replacing $A$ by $A^*$ and $D$ by $D^*$ in the above inequality we get, 
\begin{eqnarray*}
\|A+D\|^2 \leq \|A\|^2+\|D\|^2+\frac{1}{2}\|AA^*+DD^*\|+w(AD^*).
\end{eqnarray*}
This completes the proof. 
\end{proof}

Based on the above norm inequalities we obtain the following refinement of Kittaneh's inequality (\ref{imp1}).

\begin{theorem}\label{th5}
Let $A\in \mathcal{B}(\mathcal{H}).$ Then
\begin{eqnarray*}
&& \frac{1}{4}\|A^*A+AA^*\| \\
&\leq& \frac{1}{8}\bigg[ \big(\|A+A^*\|^2+\|A-A^*\|^2\big)^2 + \frac{1}{2}\big(\|A+A^*\|^2-\|A-A^*\|^2\big)^2\bigg]^{\frac{1}{2}} \leq w^2(A).
\end{eqnarray*}
\end{theorem}

\begin{proof}
From the identity (\ref{eql1}) we get,
\begin{eqnarray*}
\frac{1}{4}\|A^*A+AA^*\| &=& \frac{1}{2}\|B^2+C^2\| \\
&\leq & \frac{1}{2} \bigg[\|B\|^4+\|C\|^4+\frac{1}{2}\|B^4+C^4\|+w(B^2C^2) \bigg]^{\frac{1}{2}},\\
&& \,\,\,\,\,\,\,\,\,\,\,\,\,\,\,\,\,\,\,\,\,\,\,\,\,\,\,\,\,\,\,\,\,\,\,\,\,\,\,\,\,\,\,\,\,\,\,\,\,\,\,\,\,\,\,\,\,\,\,\,\,\,\,\,\,\,\,\,\,\,\,\,\,\,\,\,\,\,\,\,\,\,\,\,\,\,\,\,\,\,\,\,\,\,\,\,\,\,\,\,~~\mbox{by Theorem \ref{th4}}.\\
&\leq & \frac{1}{2} \bigg[\|B\|^4+\|C\|^4+\frac{1}{2}(\|B\|^4+\|C\|^4)+\|B\|^2\|C\|^2) \bigg]^{\frac{1}{2}}.
\end{eqnarray*}
This implies the first inequality of the theorem. The second inequality follows from the observation that  $\|B\|\leq w(A)$ and $\|C\|\leq w(A)$.
\end{proof}

\begin{remark}
The first inequality in Theorem \ref{th5}  is better than the first inequality in Theorem \ref{th1}. We also note that the second inequality in Theorem \ref{th5} refines the inequality (\ref{moradi}), obtained by Omidvar and Moradi  \cite[Th. 2.1]{OM}.
\end{remark}


Now, by using the inequality (\ref{bhatia}) we prove the following numerical radius inequality.

\begin{theorem}\label{th-100}
Let $A\in \mathcal{B}(\mathcal{H})$. Then
\begin{eqnarray*}
\frac{1}{4}\|A^*A+AA^*\| &\leq& \frac{1}{2}w^2(A)+ \frac{1}{8}\bigg\|(A+A^*)^2 (A-A^*)^2 \bigg\|^{1/2} \leq w^2(A).
\end{eqnarray*}

\end{theorem}

\begin{proof}
Let $A=B+{\rm i}C$ be the Cartesian decomposition of $A$.
Then from the inequality (\ref{bhatia}) we have,
\[\|B^2C^2\| \leq \frac{1}{4}\|B^2+C^2\|^2 \leq \frac{1}{4}(\|B\|^2+\|C\|^2)^2.\]
It follows from the observation $\|B\|\leq w(A)$ and $\|C\|\leq w(A)$ that
$$w^2(A)\geq \|B^2C^2\|^{1/2}\geq \| ~~|B||C|~~ \|=\|BC\|.$$
Thus, 
\[ w^2(A)\geq \frac{1}{4}\bigg\|(A+A^*)^2 (A-A^*)^2 \bigg\|^{1/2}\geq \frac{1}{4}\bigg\|(A+A^*) (A-A^*) \bigg\|.\]
This implies that 
\[ w^2(A)\geq \frac{1}{2}w^2(A)+\frac{1}{8}\bigg\|(A+A^*)^2 (A-A^*)^2 \bigg\|^{1/2}\geq \frac{1}{2}w^2(A)+\frac{1}{8}\bigg\|(A+A^*) (A-A^*) \bigg\|.\]
From \cite[Th. 2.1]{OM}, we have
\[\frac{1}{2}w^2(A)+\frac{1}{8}\bigg\|(A+A^*) (A-A^*) \bigg\|\geq \frac{1}{4}\|A^*A+AA^*\|.\]
This completes the proof of the theorem.
\end{proof}

The following corollary is obvious.

\begin{cor}
Let $A\in \mathcal{B}(\mathcal{H})$. Then
\begin{eqnarray*}
	\frac{1}{2}\|A^*A+AA^*\| -  \frac{1}{4}\bigg\|(A+A^*)^2 (A-A^*)^2 \bigg\|^{1/2}&\leq& w^2(A)  \leq 	\frac{1}{2}\|A^*A+AA^*\|.
\end{eqnarray*}
\end{cor}

\begin{remark}
We note that if  $(A+A^*)^2 (A-A^*)^2=0$ then, $w^2(A)=\frac{1}{2}\|A^*A+AA^*\|.$
\end{remark}

\section{\textbf{Operator convex function in numerical radius inequalities}}

\noindent In this section we use some important properties of operator convex functions to improve on existing numerical radius inequalities as well as norm inequalities. A real-valued continuous function $f$ on an interval $J$ is said to be operator convex if for all selfadjoint operators $A, D \in \mathcal{B}(\mathcal{H})$ whose spectra are contained in $J$ satisfy $$f \big((1 - t)A + tD\big) \leq (1 - t) f(A) + tf(D)$$ for all $t \in [0, 1]$. The function $f(t) = t^r$ is operator convex on $[0,\infty)$ if either $1\leq r\leq 2$ or $-1\leq r\leq 0$. Interested readers can go through the excellent book \cite{Bh} for further readings on operator convex functions and their characterizations.
We need the following lemmas to proceed further.

\begin{lemma}\label{lem10}
Let $A\in \mathcal{B}(\mathcal{H})$ and let $f$ be non-negative increasing operator convex function on $[0,\infty)$. Then
\begin{eqnarray}
f(w^2(A)) &\leq & \| f(\alpha |A|^2+(1-\alpha)|A^*|^2)\|,~~\forall \alpha \in [0,1].
\end{eqnarray}

\end{lemma}

\begin{proof}
Following \cite[Cor. 2.5]{BP} we have,
\begin{eqnarray}\label{eqn100}
w^2(A)\leq \|\alpha |A|^2+(1-\alpha)|A^*|^2\|, ~~\forall ~~\alpha\in [0,1].
\end{eqnarray}
Therefore, for all $\alpha \in [0,1]$,
\begin{eqnarray*}
f(w^2(A))&\leq& f(\|\alpha |A|^2+(1-\alpha)|A^*|^2\|)\leq \| f(\alpha |A|^2+(1-\alpha)|A^*|^2)\|.
\end{eqnarray*}
This completes the proof.
\end{proof}

\begin{lemma} $($\cite{D}$)$\label{lem11}
Let $f: J\rightarrow \mathbb{R}$ be an operator convex function on the interval $J$. Let $A$ and $D$ be two selfadjoint operators  with spectra in $J$. Then
\begin{eqnarray}\label{dragomir}
f\left(\frac{A+D}{2}\right) \leq \int_0^1 f \bigg( (1-t)A+tD\bigg)dt\leq \frac{1}{2}\big(f(A)+f(D)\big).     
\end{eqnarray}
\end{lemma}

If $f$ is non-negative then the  operator inequality (\ref{dragomir}) can be reduced to the following norm inequality
\begin{eqnarray}\label{2.15}
\left \| f\left(\frac{A+D}{2}\right)\right\| \leq \left \|\int_0^1 f \bigg( (1-t)A+tD\bigg)dt \right\|\leq \frac{1}{2}\big\|f(A)+f(D)\big\|.
\end{eqnarray}

\begin{lemma}$($\cite{K02}$)$\label{lem12}
Let $A,D \in \mathcal{B}(\mathcal{H})$ be positive. Then $\|A+D\|=\|A\|+\|D\|$ if and only if $\|AD\|=\|A\|\|D\|.$

\end{lemma}

We now present the first theorem of this section.

\begin{theorem}\label{th100}
Let $A\in \mathcal{B}(\mathcal{H})$ and let $f$ be non-negative increasing operator convex function on $[0,\infty)$. Then
\begin{eqnarray*}
f(w^2(A)) &\leq & \left\| \int_0^1 f \bigg( (1-t)\big(\alpha |A|^2+(1-\alpha) |A^*|^2\big)+tw^2(A)I \bigg)dt  \right\|  \\
&\leq&  \big\|f(\alpha |A|^2+(1-\alpha) |A^*|^2)\big\|, ~~\forall \alpha \in [0,1].
\end{eqnarray*}

\end{theorem}

\begin{proof}
For all $\alpha \in [0,1]$ we have,
\[\big \| \big(\alpha |A|^2+(1-\alpha) |A^*|^2\big) w^2(A)I   \big\|=\big\|\alpha |A|^2+(1-\alpha) |A^*|^2\big\| \big\|w^2(A)I\big\|.\]
Thus, it follows from Lemma \ref{lem12} that
\[ \big\|\alpha |A|^2+(1-\alpha) |A^*|^2+w^2(A)I \big\|=\big\|\alpha |A|^2+(1-\alpha) |A^*|^2\big\|+w^2(A).  \]
So, by using the inequality (\ref{eqn100}) we get,
\[w^2(A) \leq \frac{1}{2}\big\|\alpha |A|^2+(1-\alpha) |A^*|^2+w^2(A)I \big\|.\]
Then , 
\begin{eqnarray*}
f(w^2(A)) &\leq&  f\left(\frac{1}{2}\big\|\alpha |A|^2+(1-\alpha) |A^*|^2+w^2(A)I \big\|\right)\\
&\leq & \left\| f\left ( \frac{\alpha |A|^2+(1-\alpha) |A^*|^2+w^2(A)I}{2}  \right)   \right \| \\
&\leq & \left\| \int_0^1 f \bigg( (1-t)\big(\alpha |A|^2+(1-\alpha) |A^*|^2\big)+tw^2(A)I \bigg)dt  \right\|,\\
&&\,\,\,\,\,\,\,\,\,\,\,\,\,\,\,\,\,\,\,\,\,\,\,\,\,\,\,\,\,\,\,\,\,\,\,\,\,\,\,\,\,\,\,\,\,\,\,\,\,\,\,\,\,\,\,\,\,\,\,\,\,\,\,\,\,\,\,\,\,\,\,\,\,\,\,\,\,\,\,\,\,\,\,\,\,\,\,\,\,\,\,\,\,\,\,\,\,\,\,\,\,\,\,\,\,\,\,\,\mbox{by inequality  (\ref{2.15})}\\
&\leq & \frac{1}{2} \big\|f(\alpha |A|^2+(1-\alpha) |A^*|^2)+f(w^2(A))I  \big\|,~~\mbox{by inequality  (\ref{2.15})}\\
&= & \frac{1}{2} \big\|f(\alpha |A|^2+(1-\alpha) |A^*|^2)\big\| +\frac{1}{2} f(w^2(A)),~~~~\mbox{by Lemma \ref{lem12}}\\
&\leq & \big \|f(\alpha |A|^2+(1-\alpha) |A^*|^2)\big\|,~~~~\mbox{by Lemma \ref{lem10}}.
\end{eqnarray*}
This completes the proof.
\end{proof}

By considering $f(t)=t^2$ in Theorem \ref{th100}, we get the following corollary.

\begin{cor}\label{cor1}
Let $A\in \mathcal{B}(\mathcal{H})$. Then
\begin{eqnarray*}
w^2(A) &\leq & \left\| \int_0^1  \bigg( (1-t)\big(\alpha |A|^2+(1-\alpha) |A^*|^2 \big)+tw^2(A)I \bigg)^2dt  \right\|^{1/2}  \\
&\leq&  \big\| \alpha |A|^2+(1-\alpha)|A^*|^2 \big\|,~~~~\forall \alpha \in [0,1].
\end{eqnarray*}
\end{cor}

In particular, for $\alpha=\frac{1}{2}$
\begin{eqnarray*}
w^2(A) \leq  \left\| \int_0^1  \bigg( (1-t)\left( \frac{|A|^2+ |A^*|^2}{2} \right)+tw^2(A)I \bigg)^2dt  \right\|^{1/2} \leq  \frac{1}{2} \left\|  |A|^2+|A^*|^2\right \|.
\end{eqnarray*}

This inequality can be written in the following form:
\begin{eqnarray}\label{50a}
w^2(A) &\leq& \frac{1}{\sqrt{3}} \left\|  \left(\frac{|A|^2+|A^*|^2}{2}\right)^2 +w^4(A)I + w^2(A)\left(\frac{|A|^2+|A^*|^2}{2}\right) \right\|^{1/2} \\
&\leq&  \left\| \frac{|A|^2+|A^*|^2}{2}  \right\|.
\end{eqnarray}

\begin{remark}
We observe that the inequality (\ref{50a}) is sharper than the inequality (\ref{imp1}), obtained  by Kittaneh \cite[Th. 1]{kittaneh1}. We also remark that the inequality in Corollary \ref{cor1} improves on the inequality (\ref{imp3}), obtained by Bhunia and Paul \cite[Cor. 2.5]{BP}.

\end{remark}

The following theorem again involves operator convex function.

\begin{theorem}\label{th102}
Let $A\in \mathcal{B}(\mathcal{H})$ and let $f$ be non-negative increasing operator convex function on $[0,\infty)$. Then
\begin{eqnarray*}
f(w^2(A)) &\leq & \left\| \int_0^1 f \left( (1-t)\left(\alpha \left(\frac{|A|+|A^*|}{2}\right)^2+(1-\alpha) |A|^2\right)+tw^2(A)I \right)dt  \right\|  \\
&\leq&  \left\| f\left(\alpha \left(\frac{|A|+|A^*|}{2}\right)^2+(1-\alpha) |A|^2\right)\right\|, ~~\forall \alpha \in [0,1].
\end{eqnarray*}
\end{theorem}

\begin{proof}
Following \cite[Cor. 2.15]{BP} we have,
\begin{eqnarray}\label{100b}
w^2(A) &\leq&  \left\| \alpha \left(\frac{|A|+|A^*|}{2}\right)^2 +(1-\alpha)|A|^2 \right\|, ~~\forall \alpha \in [0,1].
\end{eqnarray}
Proceeding similarly  as in  Theorem \ref{th100} we get the required inequality.
\end{proof}

Considering $f(t)=t^2$ in Theorem \ref{th102}, we get the following corollary.

\begin{cor}\label{cor2b}
Let $A\in \mathcal{B}(\mathcal{H})$. Then
\begin{eqnarray*}
w^2(A) &\leq & \left\| \int_0^1  \left( (1-t)\left(\alpha \left(\frac{|A|+|A^*|}{2}\right)^2+(1-\alpha) |A|^2\right)+tw^2(A)I \right)^2dt  \right\|^{1/2}  \\
&\leq&  \left\| \alpha \left(\frac{|A|+|A^*|}{2}\right)^2 +(1-\alpha)|A|^2 \right\|, ~~\forall \alpha \in [0,1].
\end{eqnarray*}

\end{cor}

We next prove the following theorem.

\begin{theorem}\label{th101}
Let $A\in \mathcal{B}(\mathcal{H})$ and let $f$ be non-negative increasing operator convex function on $[0,\infty)$. Then
\begin{eqnarray*}
f(w^2(A)) &\leq & \left\| \int_0^1 f \left( (1-t)\left(\alpha \left(\frac{|A|+|A^*|}{2}\right)^2+(1-\alpha) |A^*|^2\right)+tw^2(A)I \right)dt  \right\|  \\
&\leq&  \left\| f\left(\alpha \left(\frac{|A|+|A^*|}{2}\right)^2+(1-\alpha) |A^*|^2\right)\right\|, ~~\forall \alpha \in [0,1].
\end{eqnarray*}
\end{theorem}

\begin{proof}
Following \cite[Cor. 2.15]{BP} we have,
\begin{eqnarray}\label{100a}
w^2(A) &\leq&  \left\| \alpha \left(\frac{|A|+|A^*|}{2}\right)^2 +(1-\alpha)|A^*|^2 \right\|, ~~\forall \alpha \in [0,1].
\end{eqnarray}
The proof then follows by using the inequality (\ref{100a}) and proceeding similarly as in Theorem \ref{th100}.
\end{proof}

By considering $f(t)=t^2$ in Theorem \ref{th101}, we get the following corollary.

\begin{cor}\label{cor2}
Let $A\in \mathcal{B}(\mathcal{H})$. Then
\begin{eqnarray*}
w^2(A) &\leq & \left\| \int_0^1  \left( (1-t)\left(\alpha \left(\frac{|A|+|A^*|}{2}\right)^2+(1-\alpha) |A^*|^2\right)+tw^2(A)I \right)^2dt  \right\|^{1/2}  \\
&\leq&  \left\| \alpha \left(\frac{|A|+|A^*|}{2}\right)^2 +(1-\alpha)|A^*|^2 \right\|, ~~\forall \alpha \in [0,1].
\end{eqnarray*}
\end{cor}
In particular, for $\alpha=1$
\begin{eqnarray*}
w^2(A) \leq  \left\| \int_0^1  \left( (1-t) \left(\frac{|A|+|A^*|}{2}\right)^2 +tw^2(A)I \right)^2dt  \right\|^{1/2} \leq  \left\| \frac{|A|+|A^*|}{2}  \right\|^2.
\end{eqnarray*}
This inequality can be written in the following form:
\begin{eqnarray}\label{50b}
w^2(A) &\leq& \frac{1}{\sqrt{3}} \left\|  \left(\frac{|A|+|A^*|}{2}\right)^4 +w^4(A)I + w^2(A)\left(\frac{|A|+|A^*|}{2}\right)^2 \right\|^{1/2} \\
&\leq&  \left\| \frac{|A|+|A^*|}{2}  \right\|^2.
\end{eqnarray}

\begin{remark}
We note that the inequality (\ref{50b}) is sharper than the inequality (\ref{imp2}), obtained Kittaneh \cite[Th. 1]{K03}. The inequality (\ref{50a}) follows from  the inequality (\ref{50b}) by using the operator convexity of the function $f(t)=t^2$. So, the inequality (\ref{50b}) is also a refinement of the inequality (\ref{imp1}), obtained by Kittaneh \cite[Th. 1]{kittaneh1}. 
\end{remark}

Next we prove the following norm inequality.

\begin{theorem}\label{norm}
Let $A,D\in \mathcal{B}(\mathcal{H})$ be positive and let $f$ be non-negative increasing operator convex function on $[0,\infty)$. Then
\begin{eqnarray*}
f(\|AD\|) \leq  \left\| \int_0^1 f \left( (1-t) \left(\frac{A+D}{2}\right)^2 +t\|AD\|I \right)dt  \right\| \leq  \left\| f\left (\left(\frac{A+D}{2}\right)^2 \right)\right\|.
\end{eqnarray*}
\end{theorem}

\begin{proof}
Using the inequality (\ref{bhatia}) and proceeding similarly as in Theorem \ref{th100}, we get the required inequality.
\end{proof}

In particular, if we consider $f(t)=t^2$ in Theorem \ref{norm},  then we get the following corollary.

\begin{cor}\label{cor5norm}
Let $A,D\in \mathcal{B}(\mathcal{H})$ be positive. Then
\begin{eqnarray*}
\|AD\| \leq  \left\| \int_0^1  \left( (1-t) \left(\frac{A+D}{2}\right)^2 +t\|AD\|I \right)^2dt  \right\|^{1/2} \leq  \frac{1}{4}\left\| A+D \right\|^2.
\end{eqnarray*}

\end{cor}

This inequality can be written in the following form:
\begin{eqnarray*}
\|AD\| \leq \frac{1}{\sqrt{3}}\left\| \left(\frac{A+D}{2}\right)^4 +\|AD\|^2I + \|AD\|\left(\frac{A+D}{2}\right)^2 \right\|^{1/2} \leq  \frac{1}{4}\left\| A+D \right\|^2.
\end{eqnarray*}

\begin{remark}
Clearly, the inequality in Corollary \ref{cor5norm} improves on the inequality (\ref{bhatia}), obtained by Bhatia and Kittaneh \cite{BK2}.

\end{remark}

The final result of this paper is an improvement of the norm inequality (\ref{300c}), obtained by Bhatia and Kittaneh   \cite{BK}. 
\begin{theorem}\label{th300}
Let $A,D\in \mathcal{B}(\mathcal{H})$ and let $f$ be non-negative increasing operator convex function on $[0,\infty)$. Then
\begin{eqnarray*}
f(\|AD^*\|) \leq  \left\| \int_0^1 f \left( (1-t) \left(\frac{|A|^2+|D|^2}{2}\right) +t\|AD^*\|I \right)dt  \right\| \leq  \left\| f\left(\frac{|A|^2+|D|^2}{2}\right) \right\|.
\end{eqnarray*}

\end{theorem}

\begin{proof}

Using the inequality (\ref{300c}) and proceeding similarly as in Theorem \ref{th100}, we get the required inequality.
\end{proof}

In particular, if we consider $f(t)=t^2$ in Theorem \ref{th300},  then we get the following corollary.

\begin{cor}\label{cor5}
Let $A,D\in \mathcal{B}(\mathcal{H})$. Then
\begin{eqnarray*}
\|AD^*\| \leq  \left\| \int_0^1  \left( (1-t) \left(\frac{|A|^2+|D|^2}{2}\right) +t\|AD^*\|I \right)^2dt  \right\|^{1/2} \leq  \frac{1}{2}\left\| |A|^2+|D|^2 \right\|.
\end{eqnarray*}

\end{cor}

This inequality can be written in the following form:
\begin{eqnarray*}
\|AD^*\| &\leq & \frac{1}{\sqrt{3}}\left\| \left(\frac{|A|^2+|D|^2}{2}\right)^2 +\|AD^*\|^2I + \|AD^*\|\left(\frac{|A|^2+|D|^2}{2}\right) \right\|^{1/2} \\
&\leq & \frac{1}{2}\left\| A^*A+D^*D \right\|.
\end{eqnarray*}

\begin{remark}
We would like to remark that the inequality in Corollary \ref{cor5} refines the  inequality (\ref{300c}) obtained by Bhatia and Kittaneh \cite{BK}.  Consider  $A=\left(\begin{array}{cc}
0 & 2\\
0 & 0
\end{array}\right)$ and $D=\left(\begin{array}{cc}
1 & 0\\
0 & 1
\end{array}\right)$. Then by elementary calculations we get,
\[\frac{1}{\sqrt{3}}\left\| \left(\frac{|A|^2+|D|^2}{2}\right)^2 +\|AD^*\|^2I + \|AD^*\|\left(\frac{|A|^2+|D|^2}{2}\right) \right\|^{1/2}=\sqrt{\frac{61}{12}}\approx 2.2546\] and \,\,\,\,$\frac{1}{2}\left\| A^*A+D^*D \right\|=\frac{5}{2}.$ This shows that  the inequality obtained in Corollary \ref{cor5} is a proper refinement of the inequality (\ref{300c}).
\end{remark}

\bibliographystyle{amsplain}

\end{document}